\documentclass[final,11pt]{elsarticle}
\pdfoutput=1
\usepackage{amsmath}
\usepackage{amsthm}
\usepackage{amssymb}
\usepackage{mathtools}
\usepackage{array}
\usepackage{booktabs}
\usepackage{bm}
\usepackage{tikz}
\usepackage{xcolor}
\usepackage{bm}

\usepackage{enumerate}
\usepackage{enumitem}
\setlist[enumerate,1]{label=(\arabic*),font=\textup,
leftmargin=7mm,labelsep=1.5mm,topsep=0mm,itemsep=-0.8mm}
\setlist[enumerate,2]{label=(\alph*),font=\textup,
leftmargin=7mm,labelsep=1.5mm,topsep=-0.8mm,itemsep=-0.8mm}

\usepackage{geometry}
\usepackage{microtype}

\usepackage[pdfstartview=FitH,bookmarksopen=false,
colorlinks=true,linkcolor=blue,citecolor=blue]{hyperref}

\newtheorem{theorem}{Theorem}[section]

\newtheorem{lemma}{Lemma}[section]

\theoremstyle{definition}
\newtheorem{definition}{Definition}[section]

\numberwithin{equation}{section}

\begin{document}

	\begin{frontmatter}
		\title{The planar Tur\'an number of $\{K_4,\Theta_5\}$\,\tnoteref{titlenote}}
		
		\author{Tao Fang}
		\ead{tao2021@shu.edu.cn}
		

		\address{Department of Mathematics, Shanghai University, Shanghai 200444, P.R. China}
				
		\begin{abstract}
		  Let $\mathcal{F}$ be a set of graphs. The planar Tur\'an number, $ex_{\mathcal{P}}(n,\mathcal{F})$, is the maximum number of edges in an $n$-vertex planar graph which does not contain any member of $\mathcal{F}$ as a subgraph. In this paper, we give upper bounds of $ex_{\mathcal{P}}(n,\{K_4,\Theta_5\})\leqslant25/11(n-2)$.  We also give constructions which show the bounds are tight for infinitely many graphs.
		\end{abstract}
		
	\end{frontmatter}
	
\section{Introduction}\label{sec1}
The concept of planar Tur\'an numbers was introduced by Dowden \cite{Dowden2016} in 2016. Let $\mathcal{F}$ be a family of graphs, and the planar Tur\'an number of $\mathcal{F}$, denoted by $ex_{\mathcal{P}}(n,\mathcal{F})$, is the maximum number of edges in a $n$-vertices planar graph  that does not contain any subgraph isomorphic to a graph in $\mathcal{F}$. When $\mathcal{F}=F$, we write it as $ex_{\mathcal{P}}(n, F)$. By the Euler's formula for planar graphs, we know that the maximum number of edges in a planar graph with $n\geqslant3$ vertices is $3n-6$. When $F$ is non-planar, $ex_{\mathcal{P}}(n, F)=3n-6$.

In 2016, Dowden \cite{Dowden2016} determined the planar Tur\'an problems of $K_3,K_4,C_4$ and $C_5$.
\begin{theorem}[\cite{Dowden2016}]
Let $n$ be a positive integer.
\begin{itemize}
\item[1.] When $n\geqslant3$, $ex_{\mathcal{P}}(n, K_3)=2n-4$;
\item[2.] When $n\geqslant3$, $ex_{\mathcal{P}}(n, K_4)=3n-6$;
\item[3.] When $n\geqslant4$, $ex_{\mathcal{P}}(n, C_4)\leqslant15(n-2)/7$;
\item[4.] When $n\geqslant11$, $ex_{\mathcal{P}}(n, C_5)\leqslant(12n-33)/5$.
\end{itemize}
\end{theorem}

In 2019, Lan et al. \cite{Lan2019T} obtained an upper bound for the planar Tur\'an number of $C_6$. This result was quickly generalized by Ghosh et al. \cite{Ghosh2022C}, where they give a sharp upper bound for $ex_{\mathcal{P}}(n, C_6)$.
\begin{theorem}[\cite{Ghosh2022C}]
Let $n~(\geqslant18)$ be a positive integer. Then
$$ex_{\mathcal{P}}(n, C_6)\leqslant \frac{5}{2}n-7.$$
\end{theorem}

A graph on at least $4$ vertices is a Theta graph if it can be obtained from a cycle by adding
an additional edge joining two non-consecutive vertices. For any positive integer $k\geqslant4$, let $\Theta_k$ be the family of non-isomorphic Theta graphs on $k$ vertices. Note that the graph families $\Theta_4$ and $\Theta_5$ contain only one graph. In this article, we use the symbols $\Theta_4$ and $\Theta_5$ to represent the unique graphs in the graph family $\Theta_4$ and $\Theta_5$. In 2019, Lan et al. \cite{Lan2019T} studied the plane Tur\'an problem of $\Theta_4$, $\Theta_5$ and $\Theta_6$, and obtained the sharp upper bound of $\Theta_4$ and $\Theta_5$.

\begin{theorem}[\cite{Lan2019T}]
Let $n$ be a positive integer.
\begin{itemize}
\item[1.] When $n\geqslant4$, $ex_{\mathcal{P}}(n, \Theta_4)\leqslant12(n-2)/5$, with equality when $n\equiv12(\text{mod}20)$;
\item[2.] When $n\geqslant5$, $ex_{\mathcal{P}}(n, \Theta_5)\leqslant5(n-2)/2$, with equality when $n\equiv50(\text{mod}120)$;
\item[3.] When $n\geqslant6$, $ex_{\mathcal{P}}(n, \Theta_6)\leqslant 18(n-2)/7$, with equality when $n=9$.
\end{itemize}
\end{theorem}

In 2020, Ghosh et al. \cite{Ghosh2020T} determined the upper bound of $\Theta_6$, and proved that the upper bound is sharp.
\begin{theorem}[\cite{Ghosh2020T}]
When $n\geqslant14$,
$$ex_{\mathcal{P}}(n, \Theta_6)\leqslant \frac{18}{7}n-\frac{48}{7}.$$
\end{theorem}

In 2023, Gy\H ori et al. \cite{Gyori2023} determined the upper bound of $\{K_4,C_5\}$ and $\{K_4,C_6\}$.
\begin{theorem}[\cite{Gyori2023}]
Let $G$ be a $\{K_4,C_5\}$-free plane graph on $n(n\geqslant15)$ vertices, then
$$e(G)\leqslant \frac{15}{7}(n-2).$$
\end{theorem}

\begin{theorem}[\cite{Gyori2023}]
Let $G$ be a $\{K_4,C_6\}$-free plane graph on $n(n\geqslant9)$ vertices, then
$$e(G)\leqslant \frac{7}{3}(n-2).$$
\end{theorem}

Inspired by their conclusions, in this paper we discuss the planar Tur\'an problem of $\{K_4,\Theta_5\}$, and our main conclusion is as follows:
\begin{theorem}\label{K_4,Theta_5}
Let $G$ be a $\{K_4,\Theta_5\}$-free plane graph on $n~(n\geqslant 25)$ vertices, then $$e(G)\leqslant\frac{25}{11}(n-2).$$
\end{theorem}

\section{Definitions and their properties}
\begin{definition}
A vertex cut of a graph $G$ is a set $S\subseteq V(G)$ such that $G-S$ has more than one component. The connectivity of $G$ is the minimum size of a vertex $S$ such that $G-S$ is disconnected or has only one vertex. A graph $G$ is $k$-connected if its connectivity is at least $k$.
\end{definition}

\begin{definition}
A block of $G$ is a maximal connected subgraph of $G$ that has no cut-vertex.
\end{definition}

\begin{definition}\label{triblk}
Let $G$ be a plane graph and $e\in E(G)$. If $e$ is not in a bounded $3$-face of $G$, we call it a trivial triangular-block. Otherwise, we recursively construct a triangular-block in the following way. Let $E(H)=\{e\}$, start with $H$ as a subgraph of $G$.
\begin{itemize}
\item[1.] Add the other edges of $3$-face containing $e$ to $E(H)$;
\item[2.] For any edge $e'$ in $E(H)$, search for a bounded $3$-face containing $e'$. Add the other edges in this bounded $3$-face to $E(H)$;
\item[3.] Repeat step $2$ until we can no longer find a bounded $3$-face containing any edges in $E(H)$.
\end{itemize}
We denote the triangular-block obtained with $e$ as the starting edge as $B(e)$.
\end{definition}

Let $G$ be a plane graph, and the triangular-block gives a useful decomposition of the graph. We have the following observations.
\begin{itemize}
\item[1.] If $H$ is not a trivial triangular-block, and $e_1,e_2\in E(H)$, then $B(e_1)=B(e_2)=H$;
\item[2.] Any two $G$ triangular-blocks of $G$ are edge disjoint.
\end{itemize}

Let $\mathcal{B}$ be a family of triangular-blocks of the graph $G$. By the observation $2$ above, we have
\begin{equation}
e(G)=\sum_{B\in\mathcal{B}}e(B), \label{e(G)e(B)}
\end{equation}
where $e(G)$ and $e(B)$ represent the number of edges of $G$ and $B$, respectively.

Let $B_k$ be a triangular-block on $k$ vertices. In 2020, Ghosh et al. \cite{Ghosh2020T} describe all possible triangular-blocks when ~$k\in\{2,3,4,5\}$~, as shown in Figure \ref{block}. There are four types of triangular-blocks on $5$ vertices. There are two types of triangular-blocks on $4$ vertices. Note that $B_{4,a}=K_4$. The $3$-vertex triangular-blocks and the $2$-vertex triangle blocks are $K_3$ and $K_2$, respectively.

\begin{figure}[!h]
\centering
\includegraphics [width=1 \linewidth] {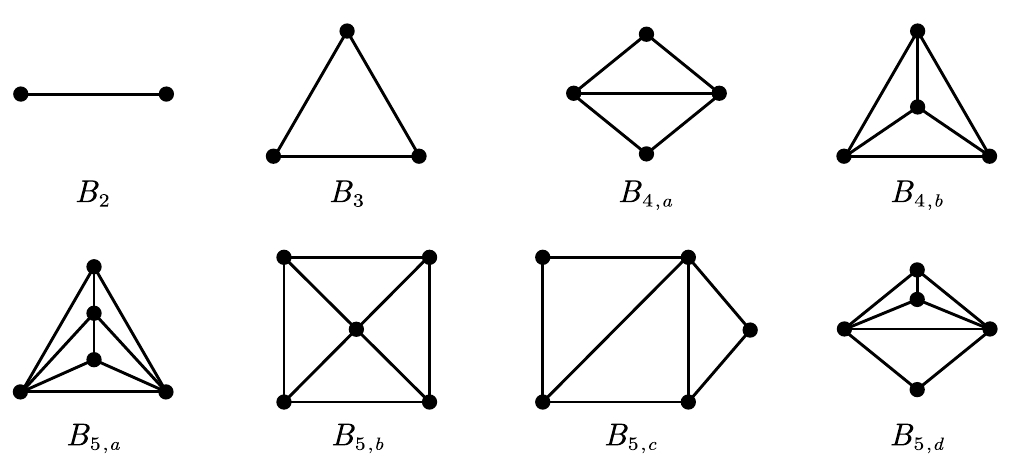}
\caption{$2,3,4,5$ vertices of triangular-blocks}
\label{block}
\end{figure}

\begin{lemma}
Let $G$ be a $\{K_4,\Theta_5\}$-free plane graph. Then $G$ contains no $k$-vertex triangular-block for all $k\geqslant5$.
\end{lemma}
\begin{proof}
Let $G$ be an $\{K_4,\Theta_5\}$-free plane graph, it is noted that $B_{5,a}$ and $B_{5,d}$ contain $K_4$ as a subgraph, and $B_{5,b}$ and $B_{5,c}$ contain $\Theta_5$ as a subgraph. Therefore, $G$ does not contain $5$-vertex triangular-blocks.

Let $B_{k-1,i}$ be a $(k-1)$-vertex triangular-block, $B_{k,j}$ be a $k$-vertex triangular-block, where $k>5$. If $G$ contain $B_{k-1,i}$ as a subgraph and there are some edges that allow step $2$ in Definition \ref{triblk} to continue, we add these edges to $E(H)$ until we cannot find a $3$-face for any edge in $E(H)$. Assuming that one of the newly added edges $e'=v_xv_y$, then $v_x,v_y$ cannot both be vertices on $B_{k-1,i}$, otherwise it would be contradictory to $B_{k-1,i}$ is a $(k-1)$-vertex triangular-block. Without loss of generality, we suppose that $v_x\notin V(B_{k-1,i})$, then $v_x$ and $V(B_{k-1,i})$ form a $k$-vertex triangular-block, so for the $k$-vertex triangular-block $B_{k,j}$, there is at least one vertex $v\in V(G)$ such that $G[B_{k,j}]/\{v\}$ is a $(k-1)$-vertex triangular-block. Therefore, $B_{k,j}$ contains at least one of $\{B_{5,a},B_{5,b},B_{5,c},B_{5,d}\}$ as subgraph.
\end{proof}

\begin{definition}
Let $G$ be a plane graph, and $e$ is an edge on the $G$. Let the two faces incident to $e$ have length $\ell_1$ and $\ell_2$, respectively. If $e$ is incident to only one face, let $\ell_1=\ell_2$ be the length of that face. The contribution of $e$ to the face number of $G$, denoted by $f(e)$, is defined as
$$f(e)=\frac{1}{\ell_1}+\frac{1}{\ell_2}.$$
Let $B\subseteq E(G)$ be a triangular-block. The contribution of $B$ to the face number of $G$, denoted by $f(B)$, is defined as
$$f(B)=\sum_{e\in B}f(e).$$
The contribution of $B$ to the edge number of $G$, denoted by $e(B)$, is defined as the size of $B$.
\end{definition}

Note that for any face $F$ on $G$, we have $\sum_{e\in E(F)}f(e)=1$. By the observation $2$ above, we have the face number of $G$ as follow:
\begin{equation}
f(G)=\sum_{B\in\mathcal{B}}f(B). \label{f(G)f(B)}
\end{equation}

\section{the planar Tur\'an number of $\{K_4,\Theta_5\}$}

\subsection{Extremal Graph Construction}

First, we construct a $\{K_4,\Theta_5\}$-free planar graph such that
\begin{lemma}
For every $n\equiv 24~mod~88$, there exits a $\{K_4,\Theta_5\}$-free planar graph $G$ with $e(G)=\frac{25}{11}(n-2)$.
\end{lemma}

This construction comes from the fact that $B_{4,b}$ is the optimal triangular-block k in
terms of edge-face ratio, see (\ref{K2<0}), (\ref{K3<0}) and (\ref{B4b<0}). In the following, we will do induction on the construction. The infrastructure $H_0$ is shown in Figure \ref{H_0}, which has $24$ vertices and $45$ edges. Replace $B_{4,a}$ with dashed edges, as shown in Figure \ref{redl}. The inductive step is shown in the Figure \ref{H_k}, we add a layer consisting of $88$ vertices and $200$ edges to $H_0$, and we get $H_k$. Finally, we get $G_k$ from $H_k$ by adding $5$ edges, as shown in the Figure \ref{G_k}. By observing the structure of $G_k$, we can see that it only contains $B_{4,b}$ as a triangular block, and each boundary edge of this triangular-block is adjacent to a face of length $5$. In fact, $G_k$ has $88k+24$ vertices and $200k+50$ edges. When $k\geqslant0$, $200k+50 = \frac{25}{11}(88k+24-2)$. Obviously, $G_k$ does not contain $K_4$ as a subgraph, nor does it contain $\Theta_5$ as a subgraph.
\begin{figure}[!h]
\centering
\includegraphics [width=0.3 \linewidth] {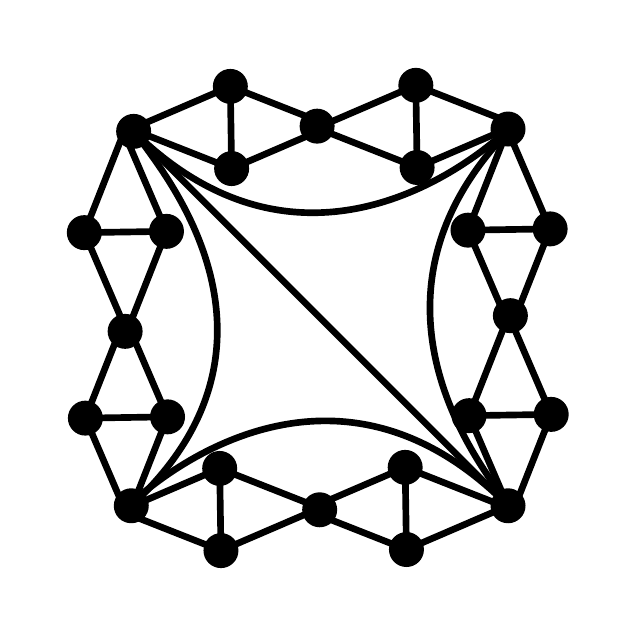}
\caption{~The base case $H_0$}
\label{H_0}
\end{figure}

\begin{figure}
\centering
\includegraphics [width=0.4 \linewidth] {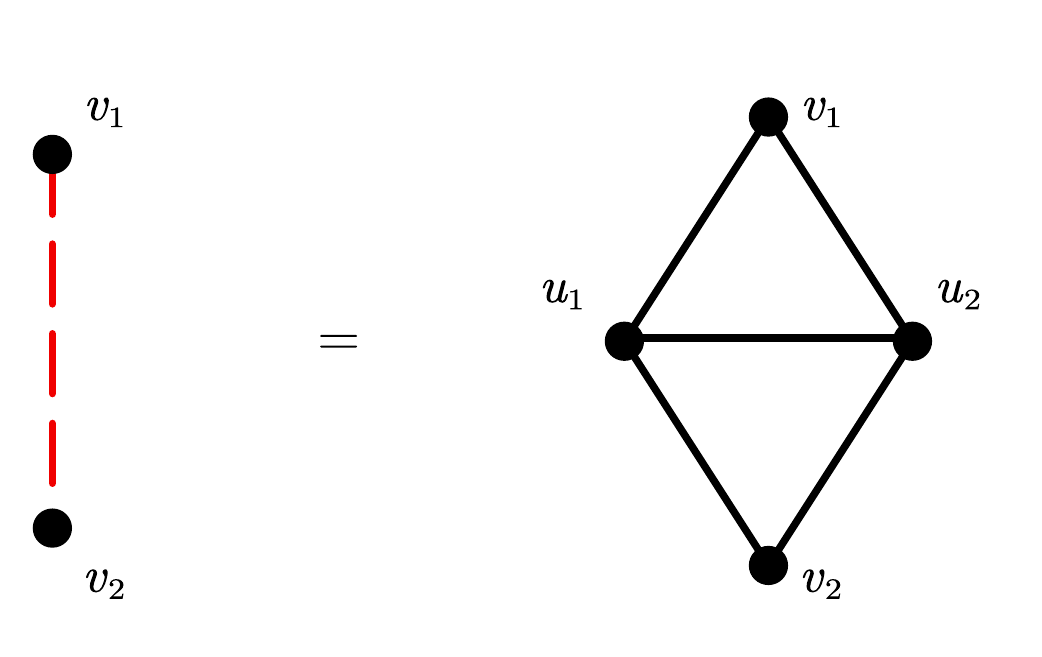}
\caption{~Replace the dashed edge with $B_{4,a}$}
\label{redl}
\end{figure}

\begin{figure}
\centering
\includegraphics [width=0.6 \linewidth] {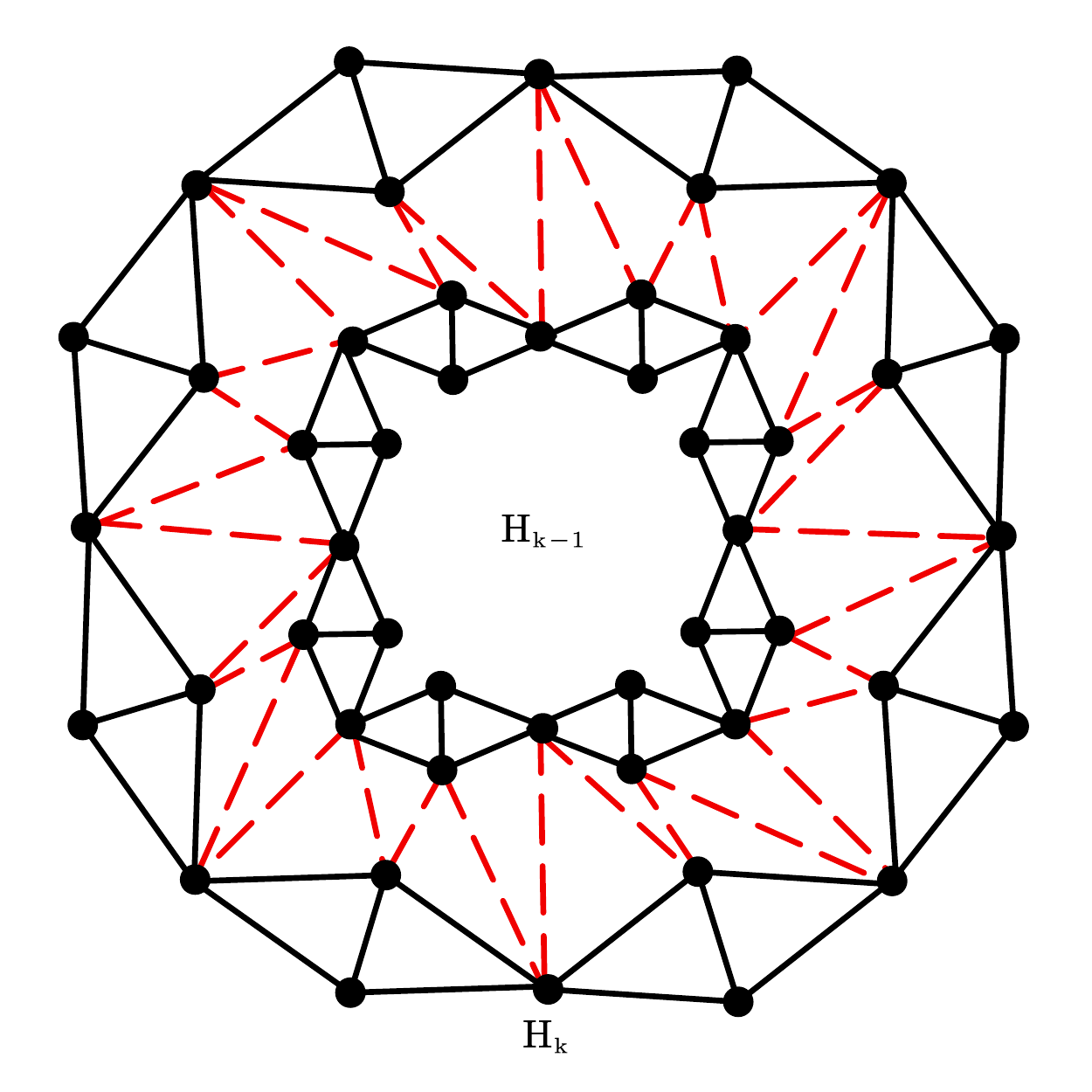}
\caption{~$H_k$ from $H_{k-1}$}
\label{H_k}
\end{figure}

\begin{figure}
\centering
\includegraphics [width=0.8 \linewidth] {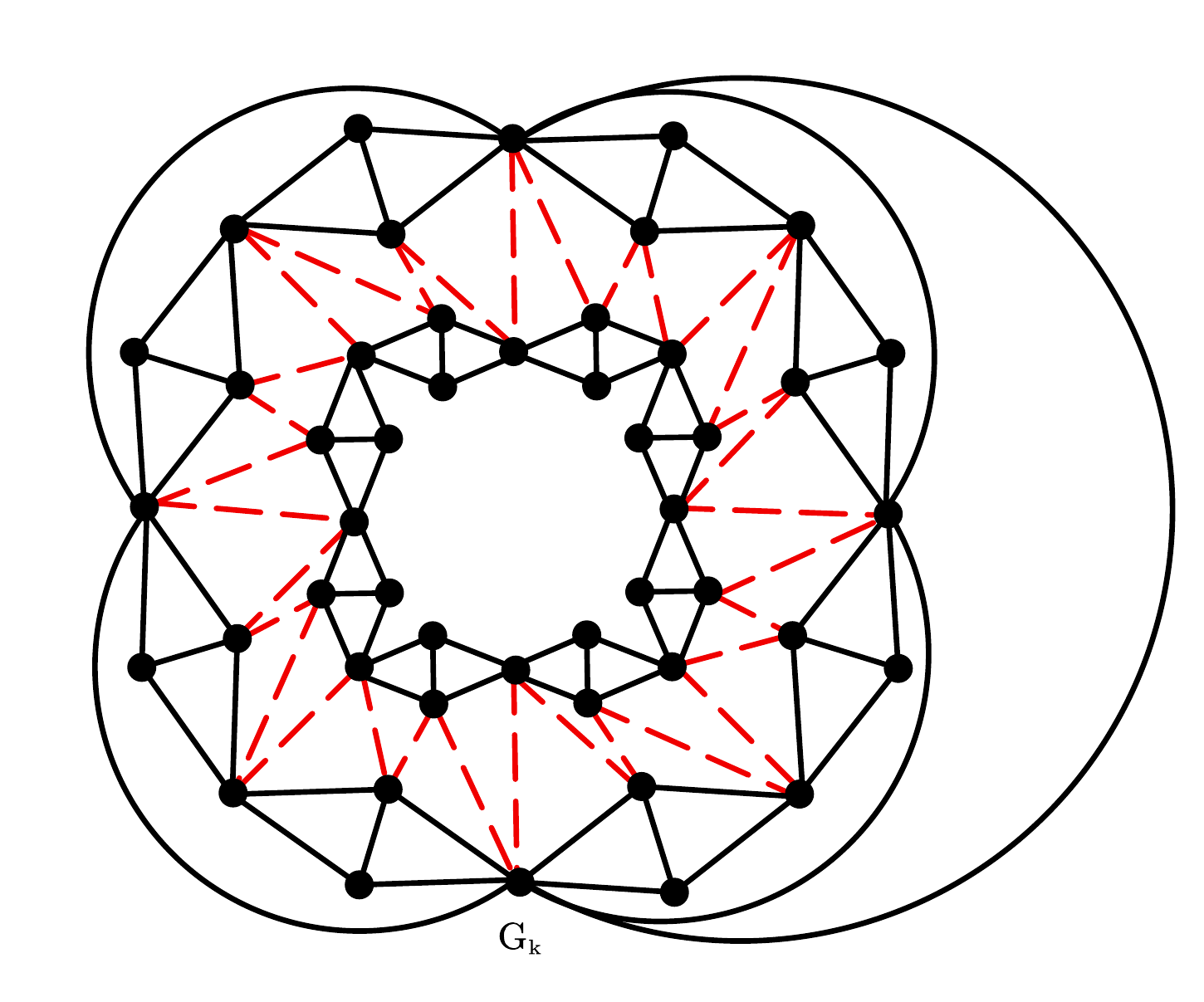}
\caption{~$G_k$ from $H_k$}
\label{G_k}
\end{figure}

\subsection{Proof of theorem \ref{K_4,Theta_5}}

We first discuss the situation where $\delta(G)\geqslant3$ is satisfied.
\begin{lemma}\label{dgeq3}
Suppose $G$ is a $\{K_4, \Theta_5\}$-free plane graph on $n~(n\geqslant5)$ vertices with $\delta(G)\geqslant3$, then
$$e(G)\leqslant \frac{25}{11}(n-2).$$
\end{lemma}
\begin{proof}
Let $G$ be a $\{K_4,\Theta_5\}$-free plane graph on $n~(n\geqslant5)$ vertices with $\delta(G)\geqslant3$. Let the number of faces of $G$ is $f(G)$. According to Euler$'$s formula, we have $e(G)\leqslant\frac{25}{11}(n-2)$ is equivalent to $25f(G )-14e(G)\leqslant0$. Then from (\ref{e(G)e(B)}), (\ref{f(G)f(B)}), we can get that $25f(G )-14e(G)=\sum_{B\in B(G)}(25f(B)-14e(B))$, where $B$ is the triangular-block of $G$. Next, we will classify different $B$ and prove that for any $B\in B(G)$, $25f(B)-14e(B)\leqslant0$.

Since $G$ is $\{K_4,\Theta_5\}$-free and $B_{4,b}=K_4$, so $B$ has only $3$ possible types, namely $B_2$, $B_3$ and $B_{4,a}$. If $B=B_2$, since $\delta(G)\geqslant3$, this edge must be incident to two faces each with length at least $4$, then
\begin{equation}\label{K2<0}
25f(B)-14e(B)\leqslant 25\cdot(\frac{1}{4}+\frac{1}{4})-14=-1.5<0.
\end{equation}
If $B=B_3$, because $G$ is $\Theta_5$-free, then no edge of $B_3$ can be incident to a face with length $4$. Then,
\begin{equation}\label{K3<0}
25f(B)-14e(B)\leqslant 25\cdot(1+\frac{3}{5})-14\cdot3=-2<0.
\end{equation}
If $B=B_{4,b}$, no edge of $B_{4,b}$ can be adjacent to a face with length $4$, see Figure \ref{B4a-4f}. So
\begin{equation}\label{B4b<0}
25f(B)-14e(B)\leqslant 25\cdot(2+\frac{4}{5})-14\cdot5=0.
\end{equation}

\begin{figure}
\centering
\includegraphics [width=0.7 \linewidth] {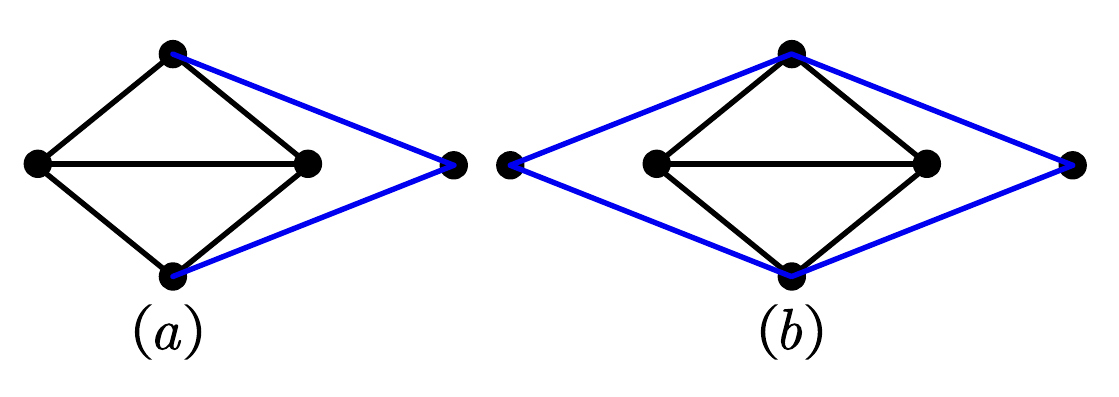}
\caption{~$4$-face be incident to $B_{4,a}$}
\label{B4a-4f}
\end{figure}

From the inequalities (\ref{K2<0}), (\ref{K3<0}) and (\ref{B4b<0}), we can get
$$25f(G)-14e(G)=\sum_{B\in B(G)}(25f(B)-14e(B))\leqslant0.$$
\end{proof}

Next, we prove the theorem \ref{K_4,Theta_5}. Let $G$ be a $\{K_4,\Theta_5\}$-free plane graph on $n~(n\geqslant25)$ vertices. We will prove that $e(G)\leqslant \frac{25}{11}(n-2)$, the above formula is equivalent to $25n-11e(G)\geqslant50$, or $n <25$. For convenience, we abbreviate $e(G)$ as $e$. For graph $G$, we define the following operation:
\begin{itemize}
    \item Delete vertex with degree at most $2$.
\end{itemize}

For the graph $G$, repeat the above operation until it can no longer go on, and the induced subgraph is $G^{\prime}$. Then $\delta(G^{\prime})\geqslant3$. Denote $|E(G^{\prime})|=e^{\prime}$, $|V(G^{\prime})|=n^{\prime}$. Each time we do this operation, denote the number of edges deleted in this process as $e_d$. Since the degree of the deleted vertex is at most $2$, so $e_d\leqslant2$ and $25\cdot1-11e_d\geqslant1 $. Therefore, we have
$$25n-11e\geqslant25n^{\prime}-11e^{\prime}+1\cdot (n-n^{\prime})$$

If $G'$ is an empty graph, since the last two vertices we delete has at most one edge between them, we have
$$25n-11e\geqslant 25\cdot2-11\cdot1+(n-2)=n+37,$$
which is greater than or equal to $50$ when $n\geqslant13$. If not, let $b$ be the number of blocks in $G^{\prime}$. In particular, let $b_2, b_3, b_4$ denote the number of blocks of size $2,3,4$ respectively. Let $b_5$ be the number of blocks of size at least $5$. Therefore we have $b=\sum^{5}_{i=2}b_i$ .
For a block of size $n^{\prime\prime}$,
when $n^{\prime\prime}\geqslant5$, since $\delta(G^{\prime})\geqslant3$, by Lemma \ref{dgeq3}, we have
$25n^{\prime\prime}-11e^{\prime\prime}-25\geqslant50-25=25$.
When $n^{\prime\prime}=4$, since the graph $G$ does not contain $K_4$ as a subgraph, as in $B_{4,a}$, we have
$25n^{\prime\prime}-11e^{\prime\prime}-25\geqslant25\cdot4-11\cdot 5-25=20$.
When $n^{\prime\prime}=3$, as in $K_3$, we have $25n^{\prime\prime}-11e^{\prime\prime}-25\geqslant25 \cdot3-11\cdot 3-25=17$.
When $n^{\prime\prime}=2$, as in $K_2$, we have
$25n^{\prime\prime}-11e^{\prime\prime}-25\geqslant25\cdot2-11\cdot 1-25=14$.

Let the number of vertices of the $i$-th block be $n_i$, and the number of edges be $e_i$. Using the lower bounds presented above, when $n\geqslant n^{\prime}\ne 0$, we have
\begin{align*}
25n^{\prime}-11e^{\prime} &\geqslant 25\left(\sum^{b}_{i=1}n_i-(b-1)\right)-11\sum^{b} _{i=1}e_i\\
&=\sum^{b}_{i=1}\left(25n_i-11e_i-25\right)+25\\
&\geqslant25b_5+20b_4+17b_3+14b_2+25.
\end{align*}
Therefore,
$$25n-11e\geqslant25b_5+20b_4+17b_3+14b_2+25+(n-n^{\prime}).$$
If $b_5\geqslant1$, we have $25n-11e\geqslant50$.
Now, we assume that $b_5=0$. It follows that $b=\sum^4_{i=2}b_i$, $n\leqslant\sum^4_{i=2}ib_i+ (n-n^{\prime}).$
Therefore,
\begin{align*}
25n-11e&\geqslant20b_4+17b_3+14b_2+25+(n-n^{\prime})\\
&\geqslant25+\sum^4_{i=2}ib_i+(n-n^{\prime})\\
&\geqslant25+n
\end{align*}
The above formula is greater than or equal to $50$ if and only if $n\geqslant25$.

\end{document}